\documentclass[a4,12pt]{amsart}
\usepackage{amsmath,amssymb}

\theoremstyle{plain}
\newtheorem*{theorem}{Theorem}
\newtheorem{proposition}{Proposition}
\newtheorem{lemma}[proposition]{Lemma}

\theoremstyle{definition}

\newcommand{\przest}{L^2(\mathbb{R}^d)}

\newcommand{\punkt}{x_1,y_1,...,x_d,y_d}
\newcommand{\hi}{\mathcal{H}}
\newcommand{\la}{\langle}
\renewcommand{\leq}{\leqslant}
\renewcommand{\geq}{\geqslant}

\DeclareMathOperator{\tr}{tr}

\begin{document}
\title[Quantum domain of attraction]{Domain of attraction of
Gaussian \\probability operators in quantum limit theory}
\author{Katarzyna Lubnauer}
\address{Faculty of Mathematics and Computer Science\\
        \L\'od\'z University \\
        ul. S. Banacha 22 \\
        90-238 \L\'od\'z, Poland}
\email{lubnauer@math.uni.lodz.pl}
\author{Andrzej \L uczak}
\email{anluczak@math.uni.lodz.pl}
\keywords{Probability operators, quantum limit theorem, domain of attraction}
\subjclass{Primary: 46L53; Secondary: 28D05}
\date{}
\begin{abstract}
 We characterise the class of probability operators belonging to
 the domain of attraction of Gaussian limits in the setup which is
 a slight generalisation of Urbanik's scheme of noncommutative
 probability limit theorems.
\end{abstract}
\maketitle
\section{Preliminaries and notation}
In our investigation of the domain of attraction in quantum limit
theory we adopt the approach introduced in the fundamental paper
\cite{U} which can be briefly described as follows. Let $\hi$ be a
separable Hilbert space. By a \emph{probability operator} we mean a
positive operator on $\hi$ of unit trace. It is well known that such
operators are in a one-to-one correspondence with normal states
$\rho$ on $\mathbb{B}(\hi)$, and this correspondence is given by the
formula
\[
 \rho(A)=\tr AT,\qquad A\in\mathbb{B}(\hi).
\]
The set of all probability operators on $\hi$ will be denoted by
$\mathfrak{P}$. By $\mathfrak{L}^1$ we shall denote the set of all
trace-class operators on $\hi$, and by \linebreak $\mathfrak{L}^2$
--- the set of all Hilbert--Schmidt operators.

Let  $z\mapsto V(z)$ be an irreducible projective unitary
representation of the group $\mathbb R^{2d}$ on $\hi$, satisfying the
Weyl--Segal commutation relations
\begin{equation}\label{W-S}
 V(z)V(z')=e^{\frac{i}{2}\Delta(z,z')}V(z+z'),
\end{equation}
where  $z, z'\in\mathbb R^{2d},\,z = (\punkt),\,z'=(x_1',
y_1',\dots,x_d',y_d')$, and
\[
\Delta(z, z') = \sum_{k=1}^d(x_ky'_k - y_kx'_k).
\]
Fix $z\in\mathbb{R}^{2d}$. It is easily seen that
$\{V(tz):t\in\mathbb{R}\}$ is a one parameter unitary group, thus by
the Stone theorem there is a selfadjoint operator $R(z)$ on $\hi$
such that
\[
 V(tz)=e^{itR(z)},\qquad t\in\mathbb{R},
\]
consequently,
\[
 V(z)=e^{iR(z)}.
\]
The operators $R(z)$ are called in \cite{H} \emph{canonical
observables}. Let
\[
 R(z)=\int_{-\infty}^{\infty}\lambda\,E_z(d\lambda)
\]
be the spectral representation of $R(z)$. For a probability operator
$T$ we define its \emph{mean value} $m_1^T(z)$, \emph{second moment}
$m_2^T(z)$ and \emph{variance} $\sigma^2_T(z)$ by the formulae
\begin{align*}
 m_1^T(z)&=\int_{-\infty}^{\infty}\lambda\,\tr TE_z(d\lambda)\\
 m_2^T(z)&=\int_{-\infty}^{\infty}\lambda^2\,\tr TE_z(d\lambda)\\
 \sigma^2_T(z)&=\int_{-\infty}^{\infty}(\lambda-m_1^T(z))^2\,\tr
 TE_z(d\lambda)=m_2^T(z)-m_1^T(z)^2
\end{align*}
(cf. \cite[Chapter V, \$ 4]{H}). Note that the notions defined above
correspond to the mean value (expectation), second moment and
variance, respectively, of the Borel probability measure $\mu_z$
determined by the formula
\begin{equation}\label{miara}
 \mu_z(\Lambda)=\tr TE_z(\Lambda),\qquad \Lambda\in\mathcal{B}(\mathbb{R}).
\end{equation}
A probability operator $T$ is said to have \emph{finite variance} if
for each $z\in\mathbb{R}^{2d}$, $\sigma_T^2(z)<\infty$ (equivalently,
$m_2^T(z)<\infty$).

For a probability operator $T$ we define its \emph{characteristic
function} $\widehat{T}\colon\mathbb{R}^{2d}\to\mathbb{C}$ as
\begin{equation}\label{ch}
 \widehat{T}(z)=\tr TV(z),\qquad z\in\mathbb{R}^{2d}.
\end{equation}
$\widehat{T}$ has the following property called
$\Delta$-\emph{positive definiteness}: for arbitrary complex numbers
$c_1,\dots,c_n$ and vectors $z_1,\dots,z_n\in\mathbb{R}^{2d}$
\[
 \sum_{j,k=1}^nc_j\bar{c}_k\widehat{T}(z_j-z_k)e^{\frac{i}{2}
 \Delta(z_j,z_k)}\geq 0.
\]
`Quantum Bochner's theorem' states that for a complex-valued function
$f\colon\mathbb{R}^{2d}\to\mathbb{C}$ we have $f=\widehat{T}$ for a
certain probability operator $T$ if and only if $f$ is
$\Delta$-positive definite, continuous at the origin and  $f(0)=1$
(cf. \cite[Chapter V, \$ 4]{H}).

It is immediately seen that for an arbitrary probability operator $T$
and an arbitrary $z_0\in\mathbb{R}^{2d}$ the function
\[
 \mathbb{R}^{2d}\ni z\mapsto e^{i\la z_0,z\rangle}\widehat{T}(z)
\]
is the characteristic function of some probability operator.

Formula \eqref{ch} for $T\in\mathfrak{L}^1$ defines a map which
extends uniquely to a linear isometry from $\mathfrak{L}^2$ onto the
space of all complex-valued square integrable with respect to
Lebesgue measure functions $f$ with the norm
\[
 \|f\|_2=\Big(\frac{1}{(2\pi)^d}\int_{\mathbb{R}^{2d}}|f(z)|^2\,dz\Big)^{1/2}
\]
(cf. \cite[Chapter V, \$ 3, Theorem 3.2]{H}).

Let $\mathfrak{A}$ be the set of all Hilbert-Schmidt operators $T$
for which $\widehat{T}$ vanishes at infinity. We define the
convolution $\star$ in $\mathfrak{A}$ by setting
\[
\widehat{T_1\star T_2} = \widehat{T}_1 \widehat{T}_2.
\]
Moreover, we put $\|T\| = \|\widehat{T}\|$. Then
\[
\|T_1\star T_2\|\leq \|T_1\| \|T_2\|,
\]
and consequently, the convolution algebra $\mathfrak{A}$ is a
Banach algebra without unit. The following inclusions hold true
\[
 \mathfrak{P}\subset\mathfrak{L}^1\subset\mathfrak{A}
 \subset\mathfrak{L}^2
\]
(cf. \cite{U}).

\section{Statement of the problem}
The general scheme of quantum limit theorems introduced in \cite{U}
is as follows. For a triangular array
$\{T_{kn}:k=1,\dots,k_n;\,n=1,2\dots\}$ of probability operators, a
norming array $\{a_{kn}:k=1,\dots,k_n;\,n=1,2,\dots\}$ of positive
numbers, and a sequence $\{z_n\}$ of elements from $\mathbb{R}^{2d}$
we form probability operators $S_n$ defined by the characteristic
functions
\begin{equation}\label{lch}
 \widehat{S}_n(z)=e^{i\la z_n,z\rangle}\prod_{k=1}^{k_n}\widehat{T}_{kn}
 (a_{kn}z),\qquad z\in\mathbb{R}^{2d}.
\end{equation}
The norming constants $a_{kn}$ should satisfy the assumption of
\emph{admissibility} which means that the maps
\[
 \mathbb{R}^{2d}\ni z\mapsto\prod_{k=1}^{k_n}\widehat{T}_k(a_{kn}z)
\]
are the characteristic functions of some probability operators for
each $n$ and \emph{any} probability operators $T_1,\dots,T_n$. Now if
\[
 \lim_{n\to\infty}\widehat{S}_n(z)=\widehat{S}(z),\qquad z\in\mathbb{R}^{2d}
\]
for some function $\widehat{S}$, then from quantum Bochner's theorem
it follows that $\widehat{S}$ is the characteristic function of some
uniquely determined probability operator $S$. In this case $S$ is
called the \emph{limit operator}. In the paper \cite{U} the class of
limit operators was described under the assumption of uniform
infinitesimality of the operators from $\mathfrak{A}$ given by the
functions
$\{\widehat{T}_{kn}(a_{kn}\cdot):k=1,\dots,k_n;\,n=1,2,\dots\}$,
analogously to the case of the classical infinitely divisible limit
laws, while in the paper \cite{L2} for the case $d=1$ norming by
arbitrary $2\times 2$ matrices was considered. We shall be concerned
with a quantum counterpart of the classical stable limit laws, i.e.
we assume that $k_n=n$ and $T_{1n}=\dots=T_{nn}=T$ for some
probability operator $T$. As for norming we adopt the above-mentioned
more general approach and as the norming matrices we take matrices
$A_n$ of the form
\begin{equation}\label{nm}
 A_n=\begin{bmatrix}
         a_1^{(n)} & 0 & \dots & 0 & 0\\
         0 & a_1^{(n)} & \dots & 0 & 0\\
         \vdots & \vdots & \ddots & \vdots\\
         0 & 0 & \dots & a_d^{(n)} & 0\\
         0 & 0 & \dots & 0 & a_d^{(n)},
         \end{bmatrix}.
\end{equation}
As in the scalar case we put the assumption of \emph{admissibility}
of the matrices $A_n$ which means that for each $n$ and \emph{any}
probability operators $T_1,\dots,T_n$ the function
\[
 \mathbb{R}^{2d}\ni z\mapsto\prod_{k=1}^n\widehat{T}_k(A_nz)=
 \prod_{k=1}^n\widehat{T}_k(a_1^{(n)}x_1,a_1^{(n)}y_1,\dots,a_d^{(n)}x_d,
 a_d^{(n)}y_d)
\]
is the characteristic function of some probability operator. In this
case the limit operator $S$ will be said to belong to the
\emph{domain of attraction of the probability operator} $T$.

To justify this approach let us look at the fundamental notion of the
(multidimensional) Schr\"{o}dinger pair of canonical observables.
Define in Hilbert space $L^2(\mathbb{R}^d)$ operators $p^{(0)}_k$ and
$q^{(0)}_k,\quad k=1,...,d,$ (called momentum and position operators,
respectively) by the formulae
\[
 \begin{aligned}
  &(p^{(0)}_k\psi)(x_1,...,x_d)=(D_k\psi)(x_1,...,x_d),\\
  &(q^{(0)}_k\psi)(x_1,...,x_d)=-ix_k\psi(x_1,...,x_d),
 \end{aligned}
\]
where $D_k$ denotes the $k$-th partial derivative. The operators
$p^{(0)}_k\text{ and }q^{(0)}_k$ are unbounded densely defined and
selfadjoint; moreover, they satisfy the commutation relations
\begin{equation}\label{ccr}
 [p^{(0)}_k,p^{(0)}_j]=[q^{(0)}_k,q^{(0)}_j]=0,
 \qquad[p^{(0)}_k,q^{(0)}_j]=-i\delta_{kj}\boldsymbol{1},
\end{equation}
where for operators $A,B$ on $\przest$
\[
 [A,B]=AB-BA,
\]
and $\boldsymbol{1}$ stands for the identity operator (observe
that since $p^{(0)}_k\text{ and }q^{(0)}_k$ are densely defined,
relations \eqref{ccr} are assumed to hold only on a dense
subspace of $\przest$).

The pair
$(p^{(0)},q^{(0)})=((p_1^{(0)},q_1^{(0)}),\dots,(p_d^{(0)},q_d^{(0)}))$
is called the Schr\"{o}\-din\-ger pair of canonical observables.
Putting
\begin{equation}\label{weyl}
 \begin{split}
 &V^{(0)}(\punkt)=\exp\Big\{i\sum^d_{k=1}\big(x_kp^{(0)}_k+y_kq^{(0)}_k
 \big)\Big\},\\&(\punkt)\in\mathbb{R}^{2d}.
 \end{split}
\end{equation}
we easily see that $z\mapsto V^{(0)}(z)$ is a projective unitary
representation of the group $\mathbb R^{2d}$ on $\hi$, satisfying the
Weyl--Segal commutation relations \eqref{W-S}. Now if $T^{(0)}$ is a
probability operator on $\przest$ (we use a superscript ${}^{(0)}$
when referring to the space $\przest$) then its characteristic
function at the point $A_nz$ for $A_n$ given by the formula
\eqref{nm} equals to
\begin{align*}
 \widehat{T}^{(0)}(A_nz)&=\tr T^{(0)}V^{(0)}(A_nz)=\tr T^{(0)}
 \exp\Big\{i\sum^d_{k=1}a_k^{(n)}\big(x_kp^{(0)}_k+y_kq^{(0)}_k\big)\Big\}\\
 &=\tr T^{(0)}
 \exp\Big\{i\sum^d_{k=1}\big[x_k(a_k^{(n)}p^{(0)}_k)
 +y_k(a_k^{(n)}q^{(0)}_k)\big]\Big\}
\end{align*}
which corresponds to the passing from the multidimensional canonical
pair $((p_1^{(0)},q_1^{(0)}),\dots,(p_d^{(0)},q_d^{(0)}))$ to the
pair
\[
 ((a_1^{(n)}p_1^{(0)},a_1^{(n)}q_1^{(0)}),\dots,
 (a_d^{(n)}p_d^{(0)},a_d^{(n)}q_d^{(0)})),
\]
i.e. each of the component pairs $(p_k^{(0)},q_k^{(0)})$ being normed
by possibly different numbers $a_k^{(n)},\,k=1,\dots,d$. It is worth
noting that in the pioneering paper \cite{CH} on quantum limit
theorems, the central limit theorem was formulated just in the
language of canonical pairs, though solely in the case $d=1$ and with
the classical scalar norming $a_1^{(n)}=\frac{1}{\sqrt{n}}$.

Coming back to our setup, we have
\begin{align*}
 \lim_{n\to\infty}\widehat{S}_n(z)&=\lim_{n\to\infty}
 e^{i\la z_n,z\rangle}\big[\widehat{T}(a_1^{(n)}x_1,a_1^{(n)}y_1,\dots,
 a_d^{(n)}x_d,a_d^{(n)}y_d)\big]^n\\&=\widehat{S}(\punkt),
 \qquad z=(\punkt)\in\mathbb{R}^{2d}.
\end{align*}
It was proved in \cite{L1} that then the limit operator $S$ must be
Gaussian, i.e. $\widehat{S}$ is the characteristic function of a
Gaussian probability distribution on $\mathbb{R}^{2d}$. In the
classical commutative situation various sufficient conditions on
belonging to the domain of attraction of a Gaussian law have been
obtained --- the most celebrated being that of finite variance as in
the Lindeberg--L\'{e}vy central limit theorem. It turns out that in
the quantum case this condition is also necessary. Namely, we shall
prove the following
\begin{theorem}
 Let $T$ be an arbitrary probability operator on $\hi$. $T$ belongs
 to the domain of attraction of a Gaussian probability operator
 if and only if $T$ has finite variance.
\end{theorem}

\section{Proofs}
We begin with a simple lemma which gives a description of the
characteristic function of Gaussian probability operators in a
particular case.
\begin{lemma}
Let
\[
 f(z)=e^{-\frac{a}{2}\|z\|^2}
\]
for some $a>0$. $f$ is the characteristic function of some Gaussian
probability operator if and only if $a\geq\frac{1}{2}$.
\end{lemma}
\begin{proof}
Observe that $f$ is the characteristic function of a Gaussian
probability measure with the covariance matrix $Q=aI$. From
\cite[Chapter V, \$\$ 4, 5]{H} (see also \cite{U}) it follows that an
arbitrary positive-definite $2d\times2d$ matrix $Q$ is the covariance
matrix of a Gaussian probability operator if and only if the
following inequality holds
\begin{equation}\label{cm}
 \la Qz,z\rangle+\la Qz',z'\rangle\geq\Delta(z,z'),\qquad z,z'\in\mathbb{R}^{2d},
\end{equation}
which in our case amounts to saying that
\[
 a\big(\|z\|^2+\|z'\|^2\big)\geq\Delta(z,z'),\qquad z,z'\in\mathbb{R}^{2d}.
\]
The inequality above may be rewritten in the form
\[
 \sum_{k=1}^d (ax_k^2+ay_k^2+a{x'_k}^2+a{y'_k}^2-x_ky'_k+y_kx'_k)\geq 0,
 \quad x_k,y_k,x'_k,y'_k\in\mathbb{R}.
\]
It is easily seen that this inequality holds if and only if for each
\linebreak $k=1,\dots,d$  and arbitrary
$x_k,y_k,x'_k,y'_k\in\mathbb{R}$ we have
\[
 ax_k^2+ay_k^2+a{x'_k}^2+a{y'_k}^2-x_ky'_k+y_kx'_k\geq 0,
\]
which, in turn, is equivalent to the positive definiteness of the
matrix
\[
 \begin{bmatrix}
  a & 0 & 0 & -\frac{1}{2}\\
  0 & a & \frac{1}{2} & 0\\
  0 & \frac{1}{2} & a & 0\\
  -\frac{1}{2} & 0 & 0 & a
  \end{bmatrix}.
\]
Since the eigenvalues of this matrix are equal to $a\pm\frac{1}{2}$
the conclusion follows.
\end{proof}
We also have the following simple property of the covariance matrix
of a Gaussian probability operator
\begin{lemma}
Let $Q$ be the covariance matrix of a Gaussian probability operator.
Then $Q$ is non-singular
\end{lemma}
\begin{proof}
Indeed, assume that $Qz'=0$ for some $0\neq z'\in\mathbb{R}^{2d}$.
Then for each fixed $z\in\mathbb{R}^{2d}$ and an arbitrary
$t\in\mathbb{R}$ we have on account of \eqref{cm}
\[
 \la Qz,z\rangle=\la Qz,z\rangle+\la Q(tz'),(tz')\rangle\geq
 \Delta(z,tz')=t\Delta(z,z'),
\]
which is clearly impossible.
\end{proof}
The following proposition provides estimation on the coefficients of
the norming matrices.
\begin{proposition}\label{P}
Let $\{A_n\}$ be an admissible sequence of matrices of the form
\eqref{nm}. Then
\[
 a_k^{(n)}\geq\frac{1}{\sqrt{n}}\qquad\text{for each}\quad k=1,\dots,d.
\]
\end{proposition}
\begin{proof}
Let $T_1=\dots=T_n=T$ be Gaussian probability operators with the
characteristic function
\[
 \widehat{T}(z)=e^{-\frac{1}{4}\|z\|^2}.
\]
Then
\begin{align*}
 \prod_{k=1}^n\widehat{T}_k(A_nz)&=
 \big[\widehat{T}(a_1^{(n)}x_1,a_1^{(n)}y_1,\dots,a_d^{(n)}x_d
 a_d^{(n)}y_d)\big]^n\\&=\exp\Big[-\frac{n}{4}\sum_{k=1}^d
 a_k^{(n)2}(x_k^2+y_k^2)\Big],
\end{align*}
which is a Gaussian probability operator with covariance matrix
\[
 Q=\frac{n}{2}
  \begin{bmatrix}
   a_1^{(n)2} & 0 & \dots & 0 & 0\\
   0 & a_1^{(n)2} & \dots & 0 & 0\\
   \vdots & \vdots & \ddots & \vdots\\
   0 & 0 & \dots & a_d^{(n)2} & 0\\
   0 & 0 & \dots & 0 & a_d^{(n)2}
  \end{bmatrix}
\]
Now the inequality \eqref{cm} takes the form
\[
 \frac{n}{2}\sum_{k=1}^d a_k^{(n)2}\big(x_k^2+y_k^2+{x'_k}^2
 +{y'_k}^2\big)\geq\sum_{k=1}^d(x_ky'_k-y_kx'_k).
\]
Putting $x'_1=-x_1,\,y'_1=y_1,\,x_k=y_k=x'_k=y'_k=0$ for
$k=2,\dots,d$ we obtain
\[
 na_1^{(n)2}(x_1^2+y_1^2)\geq 2x_1y_1
\]
which means that the matrix
\[
 \begin{bmatrix}
  na_1^{(n)2} & -1\\
  -1 & na_1^{(n)2}
 \end{bmatrix}
\]
is positive definite. Consequently,
\[
 n^2a_1^{(n)4}\geq 1,
\]
i.e.
\[
 a_1^{(n)}\geq\frac{1}{\sqrt{n}}.
\]
By the same token we obtain the required inequalities for
$k=2,\dots,d$.
\end{proof}
The next lemma is a known classical result from the theory of domains
of attraction (cf. \cite[Chapter IX, \$ 8]{F}).
\begin{lemma}\label{L}
Let $\nu$ be a probability measure belonging to the domain of
attraction of a Gaussian law, i.e. there are constants
$b_n>0,\,c_n\in\mathbb{R}$ such that
\[
 \lim_{n\to\infty}e^{itc_n}\big[\hat{\nu}(b_nt)\big]^n=
 e^{itm-\frac{1}{2}\sigma^2t^2},\qquad t\in\mathbb{R},
\]
for some $m\in\mathbb{R},\,\sigma>0$. If $b_n\geq\frac{1}{\sqrt{n}}$,
then $\nu$ has finite variance, i.e.
\[
 \int_{-\infty}^{\infty}\lambda^2\,\nu(d\lambda)<\infty.
\]
\end{lemma}
\begin{proof}
We shall follow \cite{F}. First, note that the theory of limit laws
yields $b_n\to 0$. Fix an arbitrary $x>0$ and denote
\[
 U(x)=\int_{-x}^x\lambda^2\,\nu(d\lambda).
\]
According to  formula (8.12) in \cite[Chapter IX, \$ 8, Theorem
1a]{F} we have
\[
 nb_n^2U\Big(\frac{x}{b_n}\Big)\to c,
\]
for some constant $c$. (We warn the reader that there is a difference
in the notation employed in \cite{F} and here, namely, we use $b_n$
for what in \cite{F} is denoted by $\frac{1}{a_n}$ and $c_n$ for what
in \cite{F} is denoted by $b_n$.) Since
\[
 nb_n^2\geq 1,
\]
we get
\[
 \int_{-\infty}^{\infty}\lambda^2\,\nu(d\lambda)
 =\lim_{n\to\infty}\int_{-\frac{x}{b_n}}^
 {\frac{x}{b_n}}\lambda^2\,\nu(d\lambda)
 =\lim_{n\to\infty}U\Big(\frac{x}{b_n}\Big)<\infty.
\]
\end{proof}
Now we are in a position to prove our theorem.
\begin{proof}[Proof of the Theorem. Necessity.]
Assume that for some probability operator $T$, a sequence $\{z_n\}$
of vectors from $\mathbb{R}^{2d}$ and a sequence $\{A_n\}$ of
admissible matrices of form \eqref{nm} we have
\begin{equation}\label{gr}
 \begin{aligned}
  \lim_{n\to\infty}&e^{i\la z_n,z\rangle}\big[\widehat{T}
  (a_1^{(n)}x_1,a_1^{(n)}y_1,\dots,
  a_d^{(n)}x_d,a_d^{(n)}y_d)\big]^n\\&=\widehat{S}(\punkt),
 \end{aligned}
\end{equation}
for each $z=(\punkt)\in\mathbb{R}^{2d}$. Since $S$ is Gaussian
\begin{equation}\label{S}
 \widehat{S}(z)=e^{i\la z_0,z\rangle-\frac{1}{2}\la Qz,z\rangle}
\end{equation}
for some $z_0\in\mathbb{R}^{2d}$ and covariance matrix $Q$. Let
$\mu_z$ be the probability measure defined by the formula
\eqref{miara}. Our aim consists in showing that $\mu_z$ has finite
second moment. We have
\begin{equation}\label{mi}
 \widehat{T}(tz)=\tr TV(tz)=\int_{-\infty}^{\infty}e^{it\lambda}\,
 \tr TE_z(d\lambda)=\widehat{\mu}_z(t),\qquad t\in\mathbb{R}.
\end{equation}
Fix $z=(\punkt)\in\mathbb{R}^{2d}$, and put
\begin{align*}
 &\bar{z}_1=(x_1,y_1,0,\dots,0),\, \bar{z}_2=(0,0,x_2,y_2,0,\dots,0),
 \dots,\\&\bar{z}_d=(0,\dots,0,x_d,y_d).
\end{align*}
Assume for a while that for each $k=1,\dots,d,\, \bar{z}_k\neq 0$. We
have on account of \eqref{gr}, \eqref{S} and \eqref{mi}
\begin{align*}
 \lim_{n\to\infty}&\big[\widehat{\mu}_{\bar{z}_k}(a_k^{(n)}t)\big]^n
 e^{it\la z_n,\bar{z}_k\rangle}=\lim_{n\to\infty}\Big[\widehat{T}
 (a_k^{(n)}t\bar{z}_k)\Big]^n
 e^{i\la z_n,t\bar{z}_k\rangle}\\= &e^{it\la z_0,\bar{z}_k\rangle-
 \frac{1}{2}t^2\la Q\bar{z}_k,\bar{z}_k\rangle}.
\end{align*}
From Proposition \ref{P} and Lemma \ref {L} we obtain that all the
measures $\mu_{\bar{z}_k},\,k=1,\dots,d$ have finite second moments,
\[
 m_2(\mu_{\bar{z}_k})<\infty.
\]
Of course, the same is true if $\bar{z}_k=0$, because then
$\mu_{\bar{z}_k}$ is the Dirac measure concentrated at zero.

From the commutation relations \eqref{W-S} it follows that the
unitary groups
$\{V(tz):t\in\mathbb{R}\},\,\{V(t_1\bar{z}_1):t_1\in\mathbb{R}\},\dots,
\{V(t_d\bar{z}_d):t_d\in\mathbb{R}\}$ form a commuting system of
operators; moreover,
\begin{equation}\label{R}
 e^{itR(z)}=V(tz)=V(t\bar{z}_1)\cdot\ldots\cdot V(t\bar{z}_d)
 =e^{itR(\bar{z}_1)}\cdot
 \ldots\cdot e^{itR(\bar{z}_d)}
\end{equation}
for each $t\in\mathbb{R}$. It follows that there is a spectral
measure $F$ and Borel functions $f,\,f_k,\,k=1,\dots,d$ such that
\[
 R(z)=\int_{-\infty}^{\infty}f(\lambda)\,F(d\lambda),\quad
 R(\bar{z}_k)=\int_{-\infty}^{\infty}f_k(\lambda)\,F(d\lambda),
\]
and the equality \eqref{R} yields
\[
 f(\lambda)=f_1(\lambda)+\dots+f_d(\lambda).
\]
Furthermore, substituting $t=f(\lambda)$ we obtain
\[
 R(z)=\int_{-\infty}^{\infty}f(\lambda)\,F(d\lambda)
 =\int_{-\infty}^{\infty}t\,(f\circ F)(dt),
\]
where
\[
 (f\circ F)(\Lambda)=F(f^{-1}(\Lambda)),\qquad \Lambda\in\mathcal{B}
 (\mathbb{R}).
\]
On the other hand we have
\[
 R(z)=\int_{-\infty}^{\infty}\lambda\,E_z(d\lambda)
\]
and the uniqueness of the spectral decomposition yields the equality
\[
 E_z=f\circ F.
\]
By the same token we obtain the equalities
\[
 E_{\bar{z}_k}=f_k\circ F,\qquad k=1,\dots,d.
\]
Consequently, we get
\begin{align*}
 m_2(\mu_z)&=\int_{-\infty}^{\infty}t^2\,\tr TE_z(dt)
 =\int_{-\infty}^{\infty}t^2\,\tr T(f\circ F)(dt)\\
 &=\int_{-\infty}^{\infty}f^2(\lambda)\,\tr TF(d\lambda),
\end{align*}
and analogously
\[
 m_2(\mu_{\bar{z}_k})=\int_{-\infty}^{\infty}f_k^2(\lambda)\,
 \tr TF(d\lambda),\qquad k=1,\dots,d.
\]
Finally, we have
\[
 f^2(\lambda)=\big[f_1(\lambda)+\dots+f_d(\lambda)\big]^2
 \leq d\big[f_1^2(\lambda)+\dots+f_d^2(\lambda)\big],
\]
yielding
\begin{align*}
 m_2(\mu_z)&=\int_{-\infty}^{\infty}f^2(\lambda)\,\tr TF(d\lambda)
 \leq \int_{-\infty}^{\infty}d\sum_{k=1}^d f_k^2(\lambda)\,
 \tr TF(d\lambda)\\&=d\sum_{k=1}^d\int_{-\infty}^{\infty}f_k^2(\lambda)\,
 \tr TF(d\lambda)=d\sum_{k=1}^d m_2(\mu_{\bar{z}_k})<\infty,
\end{align*}
which ends the proof of necessity.

\emph{Sufficiency.} A proof of sufficiency is essentially contained
in \cite{CH}, however, since the setup of \cite{CH} is different from
the one adopted in our work and since some considerations about
centring should be taken into account we present a short proof. Let
$T$ be a probability operator having finite variance. Take
\[
 a_1^{(n)}=\dots=a_d^{(n)}=\frac{1}{\sqrt{n}}.
\]
Then the sequence of norming matrices $\{A_n\}$ reduce to the
sequence of numbers $\big\{\frac{1}{\sqrt{n}}\big\}$, and from
\cite[Proposition 2.5]{U} it follows that this sequence is admissible
(this can also be checked straightforwardly, namely, it is to be
verified that the function
\[
 \mathbb{R}^{2d}\ni z\mapsto\prod_{k=1}^n \widehat{T}_k\Big(\frac{z}
 {\sqrt{n}}\Big)
\]
is $\Delta$-positive definite for arbitrary probability operators
$T_1,\dots,T_n$). For an arbitrary $z\in\mathbb{R}^{2d}$, let as
before $\mu_z$ be the probability measure defined by the formula
\eqref{miara}. The mean value of $\mu_z$ equals to $m_1^T(z)$;
moreover, it is pointed out in \cite[Chapter V, \$ 4]{H} that $m_1^T$
is a linear function of $z$, which can be checked using the known
formula for moments of a probability measure:
\[
 m_1^T(z)=-i\frac{d}{dt}\,\widehat{\mu}_z(t)\Big|_{t=0}=
 -i\frac{d}{dt}\,\widehat{T}(tz)\Big|_{t=0}=-i\frac{d}{dt}\,
 \tr TV(tz)\Big|_{t=0}.
\]
Consequently, there are vectors $z_n\in\mathbb{R}^{2d}$ such that
\[
 \la z_n,z\rangle=-m_1^T(z)\sqrt{n}\qquad\text{for each}\quad
 z\in\mathbb{R}^{2d}.
\]
We have
\[
 e^{it\la z_n,z\rangle}\widehat{T}\Big(\frac{t}{\sqrt{n}}z\Big)=
 e^{-itm_1^T(z)\sqrt{n}}\widehat{\mu}_z\Big(\frac{t}{\sqrt{n}}\Big).
\]
From the classical Lindeberg--L\'{e}vy central limit theorem it
follows that
\[
 \lim_{n\to\infty}e^{-itm_1^T(z)\sqrt{n}}\bigg[\widehat{\mu}_z
 \Big(\frac{t}{\sqrt{n}}\Big)\bigg]^n=e^{-\frac{1}{2}\sigma_z^2t^2}
\]
for some $\sigma_z^2>0$, which means that
\[
 \lim_{n\to\infty}e^{it\la z_n,z\rangle}\bigg[\widehat{T}
 \Big(\frac{t}{\sqrt{n}}z\Big)\bigg]^n=
 e^{-\frac{1}{2}\sigma_z^2t^2}.
\]
Putting $t=1$ we get
\[
 \lim_{n\to\infty}e^{i\la z_n,z\rangle}\bigg[\widehat{T}\Big(\frac{z}
 {\sqrt{n}}\Big)\bigg]^n=e^{-\frac{1}{2}\sigma_z^2},
\]
and the existence of the limit on the left hand side means that on
the right hand side we have the characteristic function of a Gaussian
probability operator which finishes the proof.
\end{proof}

\end{document}